\documentclass{amsart}
\usepackage{amsmath,amsthm}
\usepackage{amsfonts,amssymb}
\usepackage{accents}
\usepackage{enumerate}
\usepackage{color}
\usepackage{graphicx}
\usepackage{comment}

\newcommand{\lvt}{\left|\kern-1.35pt\left|\kern-1.3pt\left|}
\newcommand{\rvt}{\right|\kern-1.3pt\right|\kern-1.35pt\right|}

\newtheorem{thm}{Theorem}[section]
\newtheorem{cor}[thm]{Corollary}
\newtheorem{lem}[thm]{Lemma}
\newtheorem{prop}[thm]{Proposition}

\theoremstyle{remark}
\newtheorem{rem}{Remark}[section]

 \def\la{{\langle}}
 \def\ra{{\rangle}}

 \def\d{\mathrm{d}}
 
 \def\i{\mathrm{i}}

 \def\a{{\alpha}}
 \def\b{{\beta}}
 \def\g{{\gamma}}

 \def\la{{\langle}}
 \def\ra{{\rangle}}

 \def\CJ{{\mathcal J}}
 \def\CL{{\mathcal L}}
 
 \def\CP{{\mathcal P}}

 \def\CV{{\mathcal V}}
 
 \def\CW{{\mathcal W}}
 
 \def\CC{{\mathbb C}}
 
 \def\NN{{\mathbb N}}
 \def\PP{{\mathbb P}}
 
 \def\RR{{\mathbb R}}

\def\lla{\langle{\kern-2.5pt}\langle}
\def\rra{\rangle{\kern-2.5pt}\rangle}

\newcommand{\wt}{\widetilde}

\graphicspath{{./}}

\numberwithin{equation}{section}

\title[Sharp Bernstein inequalities on a trapezoid]{Weighted Bernstein-type inequalities \\ on a trapezoidal domain}
\author{Tomasz Beberok}
\thanks{Affiliation: University of Agriculture in Krakow, Poland.}
\date{}

\begin{document}

\begin{abstract}
We study weighted $L^2$ Bernstein-type inequalities for algebraic polynomials on a trapezoidal domain. Using a parametrization by the unit square, we construct a natural symmetric second-order Sturm--Liouville operator and derive exact directional Bernstein inequalities with explicit constants. We also obtain strengthened estimates for the ordinary partial derivatives. Finally, we relate the corresponding boundary factors to Dini derivatives of the Siciak extremal function of the trapezoid.
\end{abstract}

\maketitle

\section{Introduction}

The classical Bernstein inequality for an algebraic polynomial $p$ of degree at most $n$ on $[-1,1]$ is
\[
\big\|\sqrt{1-x^2}\,p'\big\|_{[-1,1]}
\le n\,\|p\|_{[-1,1]}.
\]
Its weighted $L^2$ counterpart is also explicit. If $\omega_{a,b}(x)=(1-x)^a(1+x)^b$, $a,b>-1$,
then
\begin{equation}\label{eq:intro-interval-L2}
\left\|\sqrt{1-x^2}\,p'\right\|_{L^2([-1,1],\omega_{a,b})}
\le
\sqrt{n(n+a+b+1)}\,
{\|p\|}_{L^2([-1,1],\omega_{a,b})},
\end{equation}
with equality for the Jacobi polynomial $P_n^{(a,b)}$; see \cite{GuessabMilovanovic1994,Szego}. In the Chebyshev case
$a=b=-\frac12$, inequality \eqref{eq:intro-interval-L2} becomes
\[
\left\|(1-x^2)^{1/4}p'\right\|_{L^2([-1,1])}
\le
n\left\|(1-x^2)^{-1/4}p\right\|_{L^2([-1,1])},
\]
with equality for the Chebyshev polynomial of the first kind. The role of such inequalities in polynomial approximation
is discussed, for example, in \cite{DitzianTotik}.

These one-dimensional inequalities provide the basic model for Bernstein-type
estimates on multivariate domains, where the factor multiplying a derivative reflects
the geometry of the boundary. Determining exact constants in the multivariate $L^2$
setting is substantially more delicate. For example, Kro\'o \cite{K2022} determined the
exact weighted $L^2$ Bernstein inequality on the unit ball and showed that it
depends on the parity of the degree. More recently, spectral methods based on
self-adjoint second-order differential operators have led to new sharp and strengthened
$L^2$ Bernstein inequalities on simplices and balls. Strengthened weighted Bernstein inequalities on triangles
were obtained by Xu \cite{X23} as part of a broader study of Bernstein inequalities on conic domains.

For the simplex
$
\Delta^d=\{x\in\mathbb R^d:x_i\ge0,\ |x|:=x_1+\cdots+x_d\le1\},
$
equipped with the Jacobi weight
$
W_\kappa(x)=x_1^{\kappa_1}\cdots x_d^{\kappa_d}(1-|x|)^{\kappa_{d+1}}$,
$\kappa_i>-1,
$
Ge and Xu \cite{GX} showed that the classical spectral operator admits several
different self-adjoint decompositions.
These decompositions provide \(L^2\) estimates for the partial derivatives with refined boundary factors. In particular, for \(i\ne \ell\),
\begin{equation}\label{eq:intro-simplex-partial}
\left\|
\frac{\sqrt{x_i(1-|x|)}}{\sqrt{1-x_\ell}}\,
\partial_i f
\right\|_{L^2(\Delta^d,W_\kappa)}
\le
\sqrt{n(n+\kappa_1+\ldots+\kappa_{d+1}+d)}\,
\|f\|_{L^2(\Delta^d,W_\kappa)}
\end{equation}
for every polynomial $f$ of degree at most $n$ \cite[Corollary~2.7]{GX}.
Ge and Xu also proved a stronger Bernstein inequality for the partial derivative with the factor
\[
\phi_i(x)
=
\frac{\sqrt{x_i(1-|x|)}}{\sqrt{x_i+1-|x|}},
\]
which extends to weighted $L^p$ spaces with arbitrary doubling weights \cite[Theorem~3.1]{GX}.
In the case $p=2$ this gives
\[
\|\phi_i\,\partial_i f\|_{L^2(\Delta^d,W_\kappa)}
\le c\,n\,\|f\|_{L^2(\Delta^d,W_\kappa)}.
\]
The factor $\phi_i$ has a direct pluripotential-theoretic interpretation. Indeed, if $V_{\Delta^d}$ denotes the Siciak--Zaharjuta extremal function of the simplex, then
\[
D_i^+V_{\Delta^d}(x)
=
\frac{\sqrt{x_i+1-|x|}}
{\sqrt{x_i}\sqrt{1-|x|}}
=
\frac{1}{\phi_i(x)};
\]
see \cite{Baran,B1,B2,B3}. For $d=2$, this factor coincides with the one obtained from the
spectral estimate \eqref{eq:intro-simplex-partial}. For $d\ge3$, the two families are
different; in fact, the $L^p$ factor $\phi_i$ yields a stronger partial derivative
estimate than the factors obtained from the alternative self-adjoint decompositions of
the spectral operator.

Analogous sharp $L^2$ Bernstein inequalities have been obtained for the unit ball
$ \mathbb B^d=\{x\in\mathbb R^d:\|x\|\le1\}$, equipped with the Jacobi weight
\[
W_\mu(x)=(1-\|x\|^2)^\mu,\quad \mu>-1.
\]
The classical self-adjoint representation of the spectral operator yields, in particular,
\begin{equation*}
\left\|
\sqrt{1-\|x\|^2}\,\partial_i f
\right\|_{L^2(\mathbb B^d,W_\mu)}
\le
\sqrt{n(n+2\mu+d)}\,
\|f\|_{L^2(\mathbb B^d,W_\mu)}.
\end{equation*}
In addition to the above inequality, a different self-adjoint representation obtained in \cite{BX} leads to exact
$L^2$ Bernstein inequalities involving the radial derivative
$\langle x,\nabla\rangle$ and the angular derivatives
$D_{i,j}=x_i\partial_j-x_j\partial_i$. In particular,
\[
\left\|
\frac{\sqrt{1-\|x\|^2}}{\|x\|}
\langle x,\nabla\rangle f
\right\|_{L^2(\mathbb B^d,W_\mu)}
\le
\sqrt{n(n+2\mu+d)}\,
\|f\|_{L^2(\mathbb B^d,W_\mu)}
\]
for even $n$. For odd \(n\), the above estimate also holds with the smaller constant $\sqrt{n(n+2\mu+d)-d+1}$.

The same work also contains a strengthened Bernstein inequality for the partial
derivative $\partial_i$ in the weighted $L^p$ setting, with the factor
\[
\Phi_i(x)
=
\frac{\sqrt{1-\|x\|^2}}
{\sqrt{x_i^2+1-\|x\|^2}}.
\]
The corresponding pluripotential relation on the unit ball is
\[
D_i^+V_{\mathbb B^d}(x)=\frac{1}{\Phi_i(x)}.
\]
Thus, as on the simplex, the strengthened factor associated with a coordinate
derivative is precisely the reciprocal of the corresponding Dini derivative of the
Siciak extremal function.

Beyond the simplex and the ball, Xu \cite{XuParabolic}
established sharp weighted $L^2$ Bernstein inequalities
on parabolic domains using differential equations
satisfied by orthogonal polynomial bases.

The present paper considers the trapezoid
\[
\mathrm{T}=\{(x,y)\in\mathbb R^2:0\le y\le1,\ 0\le x\le2-y\}.
\]
The situation on the trapezoid is fundamentally different from that on the
simplex and the unit ball. The classical Krall–Sheffer classification (see \cite{KrallSheffer,XuSpectral}) concerns second-order differential operators for which every element of \(\mathcal V_n\) is an eigenfunction with a common eigenvalue depending only on \(n\). The present trapezoid does not belong to these classical families.
Thus, the classical spectral method used on the simplex and the ball is not
available in the present setting. Nevertheless, the parametrization of the
trapezoid by the unit square provides a natural symmetric Sturm--Liouville
structure inherited from the one-dimensional Jacobi operators. Although the
resulting differential operator does not preserve the algebraic polynomial
spaces, its associated energy form allows us to derive exact weighted $L^2$
Bernstein inequalities for the distinguished derivatives considered in this
paper. Moreover, despite the absence of a polynomial spectral structure, the
boundary factors arising in the strengthened inequalities retain a close
connection with the pluripotential geometry of the trapezoid. The factor
associated with $\partial_x$ is, up to a constant normalization, the reciprocal
of the corresponding Dini derivative of the Siciak extremal function of
$\mathrm T$, while the analogous factor for $\partial_y$ is equivalent to the
reciprocal of the corresponding Dini derivative up to an absolute constant.
These relations connect the directional inequalities obtained from the square parametrization with the pluripotential geometry of the trapezoid.
This connection therefore persists even though the polynomial spectral
structure available on the classical domains is absent on the trapezoid.

\section{Geometry of the trapezoid}\label{sec:geom}
The geometry of $\mathrm{T}$ suggests a natural reduction to the unit square.
Indeed, for each fixed $y\in[0,1]$, the horizontal section of $\mathrm{T}$ is
the interval $[0,2-y]$. Normalizing the horizontal variable by its length
leads to the coordinates
\[
s=\frac{x}{2-y},
\qquad
t=y,
\]
and hence to the parametrization
\begin{equation}\label{eq:Psi}
\Psi:[0,1]^2\to\mathrm{T},\qquad \Psi(s,t)=((2-t)s,\ t).
\end{equation}
We write $(x,y)=\Psi(s,t)$. The Jacobian determinant of $\Psi$ is
$
J(s,t)=2-t.
$
This change of variables will be used throughout the paper. In particular,
it transforms the weighted problem on $\mathrm{T}$ into a product problem
on the unit square and provides a convenient description of the relevant
first-order derivatives.

Let $f$ be a smooth function on $\mathrm{T}$ and set
$
g=f\circ\Psi.
$
By the chain rule,
\begin{equation}\label{eq:derivs_st}
\partial_s g=(2-t)\,(\partial_x f)\circ\Psi,
\qquad
\partial_t g=-s\big((\partial_x f)\circ\Psi\big)+ \big(\partial_y f\big)\circ\Psi.
\end{equation}
Consequently,
\begin{equation}\label{eq:derivs_xy}
(\partial_x f)\circ\Psi=\frac{1}{2-t}\,\partial_s g,
\qquad
(\partial_y f)\circ\Psi=\partial_t g+\frac{s}{2-t}\,\partial_s g.
\end{equation}

The relations \eqref{eq:derivs_xy} suggest introducing a directional derivative on $\mathrm{T}$ that corresponds exactly to $\partial_t$ in the square coordinates. Define
\begin{equation}\label{eq:Dy}
D_y f:=\partial_y f-\frac{x}{2-y}\,\partial_x f.
\end{equation}
Then \eqref{eq:derivs_st}, together with $s=x/(2-y)$, immediately gives
\begin{equation}\label{lem:Dy_is_dt}
(D_yf)\circ\Psi=\partial_t(f\circ\Psi).
\end{equation}
The boundary of $\mathrm{T}$ consists of four line segments:
\[
x=0,\qquad y=0,\qquad y=1,\qquad x+y=2.
\]
The corresponding distance-like quantities are
\begin{equation}\label{eq:distances}
d_1=x,\qquad d_2=y,\qquad d_3=1-y,\qquad d_4=2-x-y.
\end{equation}

Motivated by the sharp Bernstein theory on classical domains, we introduce
boundary factors built from the quantities in \eqref{eq:distances} and adapted
to the parametrization \eqref{eq:Psi}:
\begin{equation}\label{eq:Phi_def}
\Phi_x(x,y):=\sqrt{x(2-x-y)},
\qquad
\Phi_y(x,y):=\sqrt{y(1-y)}.
\end{equation}
Under the change of variables $\Psi$, the factors in \eqref{eq:Phi_def} take a product form:
\begin{equation}\label{eq:Phi_pullback}
\Phi_x(\Psi(s,t))=(2-t)\sqrt{s(1-s)},
\qquad
\Phi_y(\Psi(s,t))=\sqrt{t(1-t)}.
\end{equation}
In particular, $\Phi_x$ separates into a $t$-dependent factor and the classical
Bernstein weight $\sqrt{s(1-s)}$, while $\Phi_y$ reduces to the standard
one–dimensional factor $\sqrt{t(1-t)}$.

We work with a natural
Jacobi-type family of weights on $\mathrm{T}$. For $\a,\b,\g,\delta>-1$, we define
\begin{equation}\label{eq:W_D}
W_{\a,\b,\g,\delta}(x,y)
:= x^\g(2-x-y)^\delta\, y^\a(1-y)^\b\, (2-y)^{-(\g+\delta+1)}.
\end{equation}
We write $W:=W_{\a,\b,\g,\delta}$ when parameters are fixed. A basic consequence of \eqref{eq:W_D} and the parametrization $\Psi$, which will be used repeatedly, is that the weighted measure on $\mathrm{T}$ is transformed into a product Jacobi measure on the unit square.

\begin{lem}\label{lem:pullback}
Let $f\in L^2(\mathrm{T},W)$. Define $g=f\circ\Psi$. Then
\begin{equation}\label{eq:pullback_identity}
\int_{\mathrm{T}} |f(x,y)|^2 \, W(x,y)\,\d x\,\d y
=
\int_0^1\int_0^1 |g(s,t)|^2\, s^\g(1-s)^\delta\, t^\a(1-t)^\b\,\d s\,\d t.
\end{equation}
\end{lem}

\begin{proof}
If $x=(2-t)s$, $y=t$, then $2-x-y=(2-t)(1-s)$ and $2-y=2-t$.  Thus,
\begin{align*}
W(\Psi(s,t))
&=
\big((2-t)s\big)^\g\big((2-t)(1-s)\big)^\delta\ t^\a(1-t)^\b\ (2-t)^{-(\g+\delta+1)}
\\&=
(2-t)^{-1}s^\g(1-s)^\delta t^\a(1-t)^\b.
\end{align*}
Since $\d x \, \d y=(2-t)\,\d s \, \d t$, we get $W(\Psi(s,t))\,\d x\,\d y = s^\g(1-s)^\delta t^\a(1-t)^\b\,\d s\,\d t$,
which yields \eqref{eq:pullback_identity}.
\end{proof}

\section{Sturm--Liouville operator on the square}\label{sec:SL_square}

Throughout this section we assume that $\alpha,\beta,\gamma,\delta>-1$. Let
$
Q=[0,1]^2$, $
w_{a,b}(r):=r^a(1-r)^b,
$ and
\[
\CW(s,t):=\CW_{\a,\b,\g,\delta}(s,t)
=w_{\g,\delta}(s)w_{\a,\b}(t).
\]
For $a,b>-1$, we consider the Jacobi Sturm--Liouville differential expression
\begin{equation}\label{eq:Js}
\CJ_r(w_{a,b},h)
:=
-\frac{1}{w_{a,b}(r)}
\frac{\d}{\d r}
\Big(w_{a,b}(r)\,r(1-r)\,h'(r)\Big),
\qquad 0<r<1.
\end{equation}
Equivalently,
\[
\CJ_r(w_{a,b},h)
=
-r(1-r)h''(r)
-\big(a+1-(a+b+2)r\big)h'(r).
\]
On the square we use the tensor-product differential expression
\begin{equation}\label{eq:Ltilde}
\CL g(s,t)
:=
\CJ_s(w_{\g,\delta},g(\bullet,t))
+
\CJ_t(w_{\a,\b},g(s,\bullet)).
\end{equation}

\begin{prop}\label{prop:energy_square_bilinear}
For $f,g\in C^2(Q)$,
\begin{equation}\label{eq:energy_square_bilinear}
\la \CL g,f\ra_{L^2(Q,\CW)}
=
\int_Q
\Big(
s(1-s)(\partial_s g)(\partial_s f)
+
t(1-t)(\partial_t g)(\partial_t f)
\Big)\CW(s,t)\,\d s\,\d t.
\end{equation}
Consequently,
\begin{equation}\label{eq:symmetry_square}
\la \CL g,f\ra_{L^2(Q,\CW)}
=
\la g,\CL f\ra_{L^2(Q,\CW)},
\end{equation}
and the differential expression $\CL$ is formally symmetric and nonnegative.
\end{prop}

\begin{proof}
We integrate by parts separately in the two variables. For the $s$-part,
\begin{align*}
\la \CJ_s(w_{\g,\delta},g(\bullet,t)),f\ra_{L^2(Q,\CW)}
&=
-\int_0^1\int_0^1
\partial_s\!\left(
w_{\g,\delta}(s)s(1-s)\partial_s g
\right)
f\,w_{\a,\b}(t)\,\d s\,\d t
\\
&=
\int_Q
s(1-s)(\partial_s g)(\partial_s f)\CW(s,t)\,\d s\,\d t.
\end{align*}
The boundary terms vanish because all endpoint exponents are positive and the functions and their first derivatives are bounded. The same argument applies to the $t$-part. Summing the two identities gives
\eqref{eq:energy_square_bilinear}; symmetry and nonnegativity follow immediately.
\end{proof}

Let $\left\{P_n^{(a,b)}(u)\right\}_{n\ge0}$ be the Jacobi polynomials orthogonal on \([-1,1]\) with respect to the weight \((1-u)^a(1+u)^b\), \(a,b>-1\), and define their shifted versions on $[0,1]$ by
\[
\mathsf P_n^{(a,b)}(u):=P_n^{(a,b)}(1-2u).
\]
The Jacobi differential equation \cite[Theorem~4.2.1]{Szego}, after the change of
variables $x=1-2u$, gives
\begin{equation}\label{eq:1d_eigs}
\CJ_u\!\left(w_{a,b},\mathsf P_n^{(a,b)}\right)
=
\lambda_n^{(a,b)}\mathsf P_n^{(a,b)}(u),
\qquad
\lambda_n^{(a,b)}=n(n+a+b+1).
\end{equation}

The eigenvalue relation \eqref{eq:1d_eigs}, together with the orthogonality of
the Jacobi polynomials, imply that, for $a,b>-1$ and $n\in\mathbb N_0$, every real univariate polynomial $p$ of
degree at most $n$ satisfies
\begin{equation}\label{eq:1d_bernstein}
\int_0^1 u(1-u)|p'(u)|^2w_{a,b}(u)\,\d u
\le
n(n+a+b+1)
\int_0^1 |p(u)|^2w_{a,b}(u)\,\d u.
\end{equation}
The constant is sharp, with equality attained for $p=\mathsf P_n^{(a,b)}$.

Let $\Pi_n$ denote the space of algebraic polynomials in two variables of
total degree at most $n$, and $\CV_n(Q,\CW)$ the space of polynomials of total degree $n$
that are orthogonal on $[0,1]^2$ with respect to $\CW$.

\begin{prop}\label{prop:2d}
For $n\in\NN_0$ and $0\le k\le n$, define
\begin{equation}\label{eq:Phi_mk}
\Phi_k^n(s,t)
:=
\mathsf P_k^{(\g,\delta)}(s)
\mathsf P_{n-k}^{(\a,\b)}(t).
\end{equation}
Then $\{\Phi_k^n:0\le k\le n\}$ is a mutually orthogonal basis of
$\CV_n(Q,\CW)$, and
\begin{equation}\label{eq:2d_eigs}
\CL\Phi_k^n
=
\Lambda_{n,k}\Phi_k^n,
\qquad
\Lambda_{n,k}
=
k(k+\g+\delta+1)
+
(n-k)(n-k+\a+\b+1).
\end{equation}
\end{prop}

\begin{proof}
The eigenvalue relation follows directly from \eqref{eq:1d_eigs}, since the two
one-dimensional Jacobi expressions act in separate variables. The orthogonality
and completeness are the standard product-weight construction; see
\cite[Proposition~2.2.1]{DunklXu}.
\end{proof}

The family
$
\{\Phi_k^n:\ n\ge0,\ 0\le k\le n\}
$
is a complete orthogonal system in $L^2(Q,\CW)$. Hence the action of $\CL$ on
polynomials determines a canonical nonnegative self-adjoint realization in
$L^2(Q,\CW)$ by spectral closure. In what follows, only its action on polynomials
and the energy identity \eqref{eq:energy_square_bilinear} will be needed.

For a polynomial $g(s,t)$, we denote by
$\deg_s g$ and $\deg_t g$ its degrees with respect to $s$ and $t$,
respectively.
\begin{thm}\label{BforSq}
Let $n\in\NN$. Then every polynomial $g\in\Pi_n$ satisfies
\begin{equation}\label{eq:Bernstein_square_energy}
\|\sqrt{s(1-s)}\,\partial_s g\|^2_{L^2(Q,\CW)}
+
\|\sqrt{t(1-t)}\,\partial_t g\|^2_{L^2(Q,\CW)}
\le
M(n,\CW)\|g\|^2_{L^2(Q,\CW)},
\end{equation}
where
\[
M(n,\CW)
=
\max\{\Lambda_{n,k}:0\le k\le n\}
=
n\bigl(n+1+\max\{\alpha+\beta,\gamma+\delta\}\bigr).
\]
The constant $M(n,\CW)$ is sharp. Equality is attained by $\Phi_0^n$ if
$\alpha+\beta>\gamma+\delta$, by $\Phi_n^n$ if
$\alpha+\beta<\gamma+\delta$, and by every nonzero element of
$\operatorname{span}\{\Phi_0^n,\Phi_n^n\}$ if
$\alpha+\beta=\gamma+\delta$.

Moreover, if $g$ is a polynomial such that $\deg_s g\le n$, then
\begin{equation}\label{eq:Bernstein_square_s}
\|\sqrt{s(1-s)}\,\partial_s g\|_{L^2(Q,\CW)}^2
\le
n(n+\g+\delta+1)\|g\|_{L^2(Q,\CW)}^2.
\end{equation}
Similarly, if $\deg_t g\le n$, then
\begin{equation}\label{eq:Bernstein_square_t}
\|\sqrt{t(1-t)}\,\partial_t g\|_{L^2(Q,\CW)}^2
\le
n(n+\a+\b+1)\|g\|_{L^2(Q,\CW)}^2.
\end{equation}
\end{thm}

\begin{proof}
Expand
\[
g=\sum_{j=0}^n\sum_{k=0}^j c_{j,k}\Phi_k^j.
\]
By \eqref{eq:energy_square_bilinear} and \eqref{eq:2d_eigs},
\[
\langle \CL g,g\rangle_{L^2(Q,\CW)}
=
\sum_{j=0}^n\sum_{k=0}^j
\Lambda_{j,k}|c_{j,k}|^2\|\Phi_k^j\|_{L^2(Q,\CW)}^2.
\]
Since $\Lambda_{j,k}$ is convex in $k$, its maximum over $0\le k\le j\le n$ is attained at an endpoint, which gives
\[
\Lambda_{j,k}\le n\bigl(n+1+\max\{\alpha+\beta,\gamma+\delta\}\bigr)=M(n,\CW).
\]
Together with \eqref{eq:energy_square_bilinear}, this proves
\eqref{eq:Bernstein_square_energy}. The equality cases follow from the
endpoint eigenfunctions described above.

Finally, \eqref{eq:Bernstein_square_s} and \eqref{eq:Bernstein_square_t}
follow by applying the one-dimensional sharp Bernstein inequality
\eqref{eq:1d_bernstein} to one variable at a time and integrating in the
other variable.
\end{proof}

\section{Sturm--Liouville operator on the trapezoid}\label{sec:SL_D}

We transfer the operator $\CL$ defined in \eqref{eq:Ltilde} from the square to $\mathrm{T}$ by conjugation with $\Psi$. Define $\wt{\CL}$ on  $C^2(\mathrm{T})$ by
\begin{equation}\label{eq:L_def}
\wt{\CL} f:=\CL(f\circ\Psi) \circ \Psi^{-1}.
\end{equation}

\begin{prop}\label{prop:energy_D}
Let $f,h\in C^2(\mathrm{T})$, and set $g=f\circ\Psi$ and
$q=h\circ\Psi$. Then
\begin{equation}\label{eq:conjugation_identity}
\la \wt{\CL}f,h\ra_{L^2(\mathrm{T},W)}
=
\la \CL g,q\ra_{L^2(Q,\CW)}.
\end{equation}
Consequently, by \eqref{eq:symmetry_square}, $\wt{\CL}$ is symmetric and
nonnegative on $C^2(\mathrm T)$ with respect to the weighted inner product.
Moreover,
\begin{align}\label{eq:energy_D}
\la \wt{\CL}f,f\ra_{L^2(\mathrm{T},W)}
=&
\int_\mathrm{T} |\Phi_x(x,y) \partial_x f(x,y)|^2\,W(x,y)\,\d x\,\d y
\\&+
\int_\mathrm{T} |\Phi_y(x,y) D_y f(x,y)|^2\,W(x,y)\,\d x\,\d y, \nonumber
\end{align}
where we recall $D_y f=\partial_y f-\frac{x}{2-y}\,\partial_x f$.
\end{prop}

\begin{proof}
Identity \eqref{eq:conjugation_identity} follows from the isometry
\eqref{eq:pullback_identity} and the definition \eqref{eq:L_def}. Together
with \eqref{eq:symmetry_square}, it gives the symmetry of $\wt{\CL}$.
To derive \eqref{eq:energy_D}, set $h=f$, so that $q=g$, and use
\eqref{eq:energy_square_bilinear}:
\[
\langle \CL g,g\rangle_{L^2(Q,\CW)}
=
\int_Q
\Big(
s(1-s)\,|\partial_s g(s,t)|^2
+
t(1-t)\,|\partial_t g(s,t)|^2
\Big)\CW(s,t)\,\d s\,\d t.
\]
By \eqref{eq:Phi_pullback}, \eqref{eq:derivs_st}, and
\eqref{lem:Dy_is_dt}, we have
\[
s(1-s) |\partial_s g|^2
=
(\Phi_x(\Psi))^2 |(\partial_x f)\circ\Psi|^2,
\]
and
\[
t(1-t) |\partial_t g|^2
=
(\Phi_y(\Psi))^2 |(D_y f)\circ\Psi|^2.
\]
Multiplying by $\CW(s,t)\,\d s\,\d t$ and using
\eqref{eq:pullback_identity} gives \eqref{eq:energy_D}. Nonnegativity follows
from the right-hand side of \eqref{eq:energy_D}.
\end{proof}

We now write $\wt{\CL}$ explicitly as a second-order operator in divergence form.

\begin{prop}\label{prop:div_form}
For $f \in C^2(\mathrm{T})$, we have
\begin{align}\label{eq:div_form}
\wt{\CL} f=&
-\frac{1}{x^\g(2-x-y)^\delta}\,\partial_x\Big(x^{\g+1}(2-x-y)^{\delta+1}\,\partial_x f\Big)
\\&-\frac{1}{y^\a(1-y)^\b}\,D_y\Big(y^{\a+1}(1-y)^{\b+1}\,D_y f\Big). \nonumber
\end{align}
\end{prop}

\begin{proof}
Let $g(s,t)=(f\circ\Psi)(s,t)=f((2-t)s,t)$. By definition,
\[
(\wt{\CL} f)\circ\Psi=\CL(f\circ\Psi)=\CJ_s\!\left(w_{\g,\delta},g(\bullet,t)\right) +\CJ_t\!\left(w_{\a,\b},g(s,\bullet)\right).
\]
We treat $\CJ_s\!\left(w_{\g,\delta},g(\bullet,t)\right)$ first. Using \eqref{eq:Js},
\[
\CJ_s\!\left(w_{\g,\delta},g(\bullet,t)\right)=-\frac{1}{w_{\g,\delta}(s)}\partial_s\big(w_{\g,\delta}(s)\,s(1-s)\,\partial_s g(s,t)\big).
\]
Since
$
\partial_s g(s,t)
=
(2-t)(\partial_x f \circ\Psi )(s,t),
$
\[
s^{\gamma+1}(1-s)^{\delta+1}\partial_s g(s,t)
=
(2-t)\,
s^{\gamma+1}(1-s)^{\delta+1}
(\partial_x f \circ\Psi )(s,t).
\]
Let $\tilde{f}(x,y):=(2-y)\frac{x^{\gamma+1}(2-x-y)^{\delta+1}}{(2-y)^{\gamma+\delta+2}} \partial_x f(x,y)$. Thus
\[
s^{\gamma+1}(1-s)^{\delta+1}\partial_s g(s,t)
= (\tilde{f} \circ\Psi)(s,t).
\]
Differentiating once more with respect to $s$, we obtain
\[
\partial_s
\left(
s^{\gamma+1}(1-s)^{\delta+1}\partial_s g
\right)
=
(2-t)
\Big[
\partial_x
\Big(
\frac{x^{\gamma+1}(2-x-y)^{\delta+1}}{(2-y)^{\gamma+\delta+1}} \partial_x f
\Big)
\Big]\circ\Psi .
\]
Substituting into the definition of the $s$-component yields
\begin{align*}
\CJ_s\!\left(w_{\g,\delta},g(\bullet,t)\right)
&=
-(2-t)
\frac{1}{s^\gamma(1-s)^\delta}
\Big[
\partial_x
\Big(
\frac{x^{\gamma+1}(2-x-y)^{\delta+1}}{(2-y)^{\gamma+\delta+1}} \partial_x f
\Big)
\Big]\circ\Psi \\ &= \left[-\frac{1}{x^\g(2-x-y)^\delta}\,\partial_x\Big(x^{\g+1}(2-x-y)^{\delta+1}\,\partial_x f \Big) \right] \circ\Psi.
\end{align*}
Hence
\[
\left[\CJ_s\!\left(w_{\g,\delta},g(\bullet,t)\right)\right]\circ\Psi^{-1}(x,y)
=
-\frac{1}{x^\g(2-x-y)^\delta}\,\partial_x\Big(x^{\g+1}(2-x-y)^{\delta+1}\,\partial_x f\Big).
\]
Similarly, using \eqref{eq:Js} and \eqref{lem:Dy_is_dt}, we obtain
\[
\left[\CJ_t\!\left(w_{\a,\b},g(s,\bullet)\right)\right]\circ\Psi^{-1}(x,y)
=
-\frac{1}{y^\a(1-y)^\b}\,D_y\Big(y^{\a+1}(1-y)^{\b+1}\,D_y f\Big).
\]
This completes the proof of \eqref{eq:div_form}.
\end{proof}

We next introduce an orthogonal polynomial basis on $\mathrm{T}$
adapted to $\Psi$, following the construction from one-variable
orthogonal systems described in \cite{Koornwinder1975,DunklXu}. The factor $(2-y)^k$ appearing below removes the denominator
introduced by $x/(2-y)$ and ensures that the resulting functions are algebraic
polynomials in $(x,y)$. For each $k \in \NN_0$ let $\big\{\mathsf{P}_{m,k}^{(\a,\b)}\big\}^\infty_{m=0}$ denote the system of orthonormal polynomials on the interval $[0,1]$ with respect to the weight function $(2-y)^{2k} w_{\a,\b}(y)$.
For $n\in\NN_0$ and $0\le k\le n$, define
\begin{equation}\label{eq:Pmk}
\PP^n_k(x,y)
:=
(2-y)^k\,\mathsf{P}_k^{(\g,\delta)}\!\left(\frac{x}{2-y}\right)\,
\mathsf{P}_{n-k,k}^{(\a,\b)}(y).
\end{equation}
By definition $\PP^n_k$ is a polynomial of total degree $n$. Denote by $\CV_n(\mathrm{T},W)$ the space of polynomials of total degree $n$ that are orthogonal on $\mathrm{T}$ with respect to $W$.

\begin{prop}\label{thm:diag_D}
The family $\{\PP^n_k : 0 \leq k \leq n\}$ is a mutually orthogonal basis of $\CV_n(\mathrm T,W)$.
Moreover,
\begin{align*}
\wt{\CL} \PP^n_0=n(n+\a + \b + 1)\,\PP^n_0.
\end{align*}
\end{prop}

\begin{proof}
By construction,
$
\PP^n_0(\Psi(s,t))=\mathsf{P}_0^{(\g,\delta)}\!\left(s\right)\,
\mathsf{P}_{n,0}^{(\a,\b)}(t)
$.
Then, using \eqref{eq:Js} and Proposition  \ref{prop:2d},
\begin{align*}
(\wt{\CL} \PP^n_{0})\circ\Psi= \CL(\PP^n_{0}\circ\Psi)=\mathsf{P}_{0}^{(\g,\delta)}(s) \,\CJ_t \left(w_{\a,\b},\mathsf{P}_{n,0}^{(\a,\b)}\right)=\lambda^{(\a,\b)}_n
\PP^n_{0}\circ\Psi.
\end{align*}
Orthogonality and completeness follow from Proposition 2.6.1 of \cite{DunklXu}.
\end{proof}
\begin{rem}
In general, the polynomials
$\PP_k^n$ are not eigenfunctions of $\wt{\CL}$. The endpoint $k=0$ is an
exception, since $\mathsf P_{n,0}^{(\a,\b)}$ is, up to normalization, the
shifted Jacobi polynomial $\mathsf P_n^{(\a,\b)}$.
For $k=n$, however,
\[
(\PP_n^n\circ\Psi)(s,t)
=
c_n(2-t)^n\mathsf P_n^{(\g,\delta)}(s),
\]
where $c_n\ne0$ is a constant, and the $t$-part of $\CL$ acts nontrivially
on $(2-t)^n$. Consequently, $\PP_n^n$ is not an eigenfunction of
$\wt{\CL}$ for $n\ge1$.

More generally, $\wt{\CL}$ does not preserve the algebraic polynomial spaces
$\Pi_n$. For instance,
\begin{equation*}
\wt{\CL}x
=
(\g+\delta+2)x-(\g+1)(2-y)
+
\frac{x}{2-y}
\big(\a+1-(\a+\b+2)y\big),
\end{equation*}
which is not a polynomial for $\a,\b>-1$. Thus, unlike on the simplex and
the unit ball, the sharp Bernstein inequalities below do not arise from a spectral decomposition of $\wt{\CL}$.

The two endpoint elements $\PP_0^n$ and $\PP_n^n$ of the orthogonal basis
\eqref{eq:Pmk} will nevertheless play a distinguished role. Their behavior
with respect to $\wt{\CL}$ is essentially different: $\PP_0^n$ is an
eigenfunction of $\wt{\CL}$, whereas $\PP_n^n$ not an eigenfunction for $n \geq 1$.
Nevertheless, both endpoint polynomials turn out to be extremizers:
$\PP_0^n$ for the inequality involving $\Phi_yD_y$ and $\PP_n^n$ for the
inequality involving $\Phi_x\partial_x$. Hence the sharpness of the two
inequalities cannot be attributed solely to the spectral properties of
$\wt{\CL}$. Instead, the extremal behavior is inherited separately from
the one-dimensional Jacobi structure in the $t$- and $s$-variables,
respectively, under the parametrization $\Psi$.
\end{rem}

\section{Sharp weighted \(L^2\) Bernstein inequalities on T}
The one-dimensional sharp Bernstein inequalities for Jacobi polynomials
can now be transferred to $\mathrm{T}$ through the parametrization $\Psi$.
Indeed,
\[
(\Phi_x\partial_x f)\circ\Psi
=
\sqrt{s(1-s)}\,\partial_s(f\circ\Psi),
\qquad
(\Phi_yD_y f)\circ\Psi
=
\sqrt{t(1-t)}\,\partial_t(f\circ\Psi).
\]
This leads to the following sharp inequalities.

\begin{prop}\label{prop:bernstein_T}
Let $\a,\b,\g,\delta>-1$ and $n\in\NN$. Then, for every $f\in\Pi_n$,
\begin{align}
\|\Phi_x\partial_x f\|_{L^2(\mathrm{T},W)}^2
&\le
n(n+\g+\delta+1)
\|f\|_{L^2(\mathrm{T},W)}^2,
\label{eq:bernstein_T_x}
\\[0.5em]
\|\Phi_yD_y f\|_{L^2(\mathrm{T},W)}^2
&\le
n(n+\a+\b+1)
\|f\|_{L^2(\mathrm{T},W)}^2.
\label{eq:bernstein_T_y}
\end{align}
Both constants are sharp. Equality in \eqref{eq:bernstein_T_x} is attained
by $\PP_n^n$, whereas equality in \eqref{eq:bernstein_T_y} is attained
by $\PP_0^n$.
\end{prop}

\begin{proof}
Let $g=f\circ\Psi$. Since $f\in\Pi_n$, we have
$
\deg_s g\le n,
\deg_t g\le n.
$
By \eqref{eq:Phi_pullback}, \eqref{eq:derivs_st}, \eqref{lem:Dy_is_dt}, and the pullback identity,
\[
\|\Phi_x\partial_x f\|_{L^2(\mathrm{T},W)}^2
=
\|\sqrt{s(1-s)}\,\partial_s g\|_{L^2(Q,\CW)}^2
\]
and
\[
\|\Phi_yD_y f\|_{L^2(\mathrm{T},W)}^2
=
\|\sqrt{t(1-t)}\,\partial_t g\|_{L^2(Q,\CW)}^2.
\]
The two inequalities now follow directly from
\eqref{eq:Bernstein_square_s} and \eqref{eq:Bernstein_square_t},
respectively.

It remains to verify sharpness. From \eqref{eq:Pmk},
\[
(\PP_n^n\circ\Psi)(s,t)
=
(2-t)^n
\mathsf P_n^{(\g,\delta)}(s)
\mathsf P_{0,n}^{(\a,\b)}(t).
\]
Since $\mathsf P_{0,n}^{(\a,\b)}$ is a nonzero constant, for every fixed
$t$ the dependence on $s$ is a scalar multiple of
$\mathsf P_n^{(\g,\delta)}(s)$. Equality therefore holds in the one-dimensional Bernstein inequality
for Jacobi polynomials used above, and hence in \eqref{eq:bernstein_T_x}. Thus the constant
$n(n+\g+\delta+1)$ is sharp. Similarly,
$
(\PP_0^n\circ\Psi)(s,t)
=
\mathsf P_0^{(\g,\delta)}(s)
\mathsf P_{n,0}^{(\a,\b)}(t).
$
For $k=0$, the weight defining $\mathsf P_{n,0}^{(\a,\b)}$ is simply
$w_{\a,\b}$; hence $\mathsf P_{n,0}^{(\a,\b)}$ is, up to normalization,
the shifted Jacobi polynomial of degree $n$. Therefore equality is attained in the one-dimensional sharp Bernstein inequality for Jacobi polynomials in the \(t\)-variable, and consequently also in \eqref{eq:bernstein_T_y}. This proves the sharpness of the constant \(n(n+\a+\b+1)\).
\end{proof}

The sharp estimates above are expressed in terms of the weight $W$, which contains
the factor $(2-y)^{-1}$ arising from the Jacobian of the parametrization $\Psi$.
It is natural to examine the corresponding estimate after removing this factor
from the measure, in particular with a view toward the unweighted case.
The boundary factor that appears in this reformulation will also be of
independent interest, as it will later be shown to have a nontrivial
interpretation in terms of the Siciak extremal function of $\mathrm{T}$.

\begin{cor}\label{cor}
Let $\a,\b,\g,\delta>-1$ and $n\in\NN$. Define
\[
\widetilde W(x,y)
:=
(2-y)W_{\a,\b,\g,\delta}(x,y)
=
x^\g(2-x-y)^\delta y^\a(1-y)^\b
(2-y)^{-(\g+\delta)}
\]
and
\[
\widetilde\Phi_x(x,y)
:=
\frac{\Phi_x(x,y)}{\sqrt{2-y}}
=
\sqrt{\frac{x(2-x-y)}{2-y}}.
\]
Then, for every $f\in\Pi_n$,
\begin{equation}\label{DxT}
\|\widetilde\Phi_x\partial_xf\|_{L^2(\mathrm{T},\widetilde W)}
\le
\sqrt{n(n+\g+\delta+1)}\,
\|\widetilde f\|_{L^2(\mathrm{T},\widetilde W)},
\end{equation}
where
\[
\widetilde f(x,y):=\frac{f(x,y)}{\sqrt{2-y}}.
\]
The constant is sharp, and equality is attained for $f=\PP_n^n$.
\end{cor}

\begin{rem}
It is worth noting that the sharpness of \eqref{DxT}  is obtained
after the factor $(2-y)^{-1/2}$ is incorporated into the function. If one
keeps the polynomial $f$ itself on the right-hand side, the same argument
still yields a Bernstein inequality, but it does not determine its optimal
constant.

Indeed, since
$
\frac{1}{\sqrt{2}}\,
\|f\|_{L^2(\mathrm{T},\wt W)}
\le
\|\wt f\|_{L^2(\mathrm{T},\wt W)}
\le
\|f\|_{L^2(\mathrm{T},\wt W)},
$
Corollary~\ref{cor} gives
\[
\sqrt{\frac{n(n+\g+\delta+1)}{2}}
\le
\sup_{\substack{f\in\Pi_n\\ f\neq0}}
\frac{
\|\wt\Phi_x\partial_xf\|_{L^2(\mathrm{T},\wt W)}
}{
\|f\|_{L^2(\mathrm{T},\wt W)}
}
\le
\sqrt{n(n+\g+\delta+1)}.
\]
Moreover, the extremal polynomial $\PP_n^n$ from Corollary~\ref{cor}
satisfies
\[
\lim_{n\to\infty}
\frac{
\|\wt\Phi_x\partial_x\PP_n^n\|_{L^2(\mathrm{T},\wt W)}
}{
\sqrt{n(n+\g+\delta+1)}
\|\PP_n^n\|_{L^2(\mathrm{T},\wt W)}
}
=
\frac{1}{\sqrt2}.
\]
Thus, when measured against the norm of the polynomial itself, the extremizer
of the weighted formulation asymptotically attains the lower bound in the
above two-sided estimate. The last limit follows from
Corollary~\ref{cor} and the following lemma.
\end{rem}

\begin{lem}\label{lem:weighted-extremizer-limit}
Let
\[
\wt{\PP_n^{n}}(x,y):=\frac{\PP_n^{n}(x,y)}{\sqrt{2-y}}.
\]
Then
\[
\lim_{n\to\infty}
\frac{\big\|\wt{\PP_n^{n}}\big\|^2_{L^2(\mathrm{T},\wt W)}}
{\|\PP_n^{n}\|^2_{L^2(\mathrm{T},\wt W)}}
=
\lim_{n\to\infty}
\frac{\displaystyle\int_0^1(2-y)^{2n}w_{\a,\b}(y)\,\d y}
{\displaystyle\int_0^1(2-y)^{2n+1}w_{\a,\b}(y)\,\d y}
=
\frac12.
\]
\end{lem}
\begin{proof}
Using \eqref{eq:Pmk} and the change of variables \eqref{eq:Psi}, the factors
depending on $s$ cancel in the quotient of the two norms, which gives the
integral representation in the statement.

For fixed $\varepsilon\in(0,1)$,
\[
\frac12\le \frac{1}{2-y}\le \frac{1}{2-\varepsilon}
\qquad\text{for }y\in[0,\varepsilon],
\quad\text{and}\quad
\frac{1}{2-\varepsilon}\le \frac{1}{2-y}\le1
\qquad\text{for }y\in[\varepsilon,1].
\]
Set
\[
R_n:=
\frac{\big\|\wt{\PP_n^{n}}\big\|^2_{L^2(\mathrm{T},\wt W)}}
{\|\PP_n^{n}\|^2_{L^2(\mathrm{T},\wt W)}}.
\]
Then
\[
\frac12
\le R_n
\le
\frac{1}{2-\varepsilon}
+
\frac{\displaystyle\int_\varepsilon^1(2-y)^{2n}w_{\a,\b}(y)\,\d y}
{\displaystyle\int_0^1(2-y)^{2n+1}w_{\a,\b}(y)\,\d y}.
\]
For any fixed $\eta\in(0,\varepsilon)$,
\[
0\le
\frac{\displaystyle\int_\varepsilon^1(2-y)^{2n}y^\a(1-y)^\b\,\d y}
{\displaystyle\int_0^1(2-y)^{2n+1}y^\a(1-y)^\b\,\d y}
\le
\left(\frac{2-\varepsilon}{2-\eta}\right)^{2n}
\frac{\displaystyle\int_\varepsilon^1y^\a(1-y)^\b\,\d y}
{\displaystyle\int_0^\eta(2-y)y^\a(1-y)^\b\,\d y}
\longrightarrow0.
\]
Thus,
\[
\frac12
\le
\liminf_{n\to\infty}R_n
\le
\limsup_{n\to\infty}R_n
\le
\frac{1}{2-\varepsilon}.
\]
Letting $\varepsilon\downarrow0$, we conclude that $R_n\to\frac12$.
\end{proof}
We next turn to the ordinary derivative $\partial_y$. In this direction the
geometry of $\mathrm{T}$ is more subtle, since the relevant boundary behavior
changes across the line $x=1$. Accordingly, the natural estimate involves a
piecewise defined weight and an additional first-order term supported on the
triangular region $\{(x,y)\in\mathrm{T}:x\ge1\}$. This leads to a strengthened
Bernstein inequality for $\partial_y$ with the boundary factor
$\widetilde{\Phi}_y$ defined below.

The appearance of $\widetilde{\Phi}_y$ is particularly relevant for the
subsequent pluripotential interpretation. Although its relation with the
Siciak extremal function is not an exact identity as in the $x$-direction,
$\widetilde{\Phi}_y$ is comparable, up to an absolute constant, with the
reciprocal of the corresponding Dini derivative.
\begin{prop}\label{prop:bernstein_y}
Let $\a > -1$, and for $(x,y) \in \mathrm{T}$ let
\[
Y = Y_{\a}(x,y)=
\begin{cases}
y^\a (1-y), & \text{if } x\le 1,\\[4pt]
y^\a , & \text{if } 1 < x.
\end{cases}
\]
Define, for $(x,y) \in \mathrm{T}$,
\begin{equation*}
\wt{\Phi}_y(x,y):= \min \left\{\sqrt{y(1-y)}, \frac{\sqrt{y}\sqrt{2 - x - y} }{\sqrt{2 - x}} \right\}.
\end{equation*}
Let $\triangle=\{(x,y) \in \mathbb{R}^2 : 0 \leq y \leq 1, \, 1 \leq x \leq 2-y\}$. Then, for each $f \in \Pi_n$, we have
\begin{equation}\label{DyT}
\|Hf\|^2_{L^2(\triangle,Y)} + \|\wt{\Phi}_y \partial_y f\|^2_{L^2(\mathrm{T},Y)} \leq n(n+\a +2) {\| f\|}^2_{L^2(\mathrm{T},Y)},
\end{equation}
where
\begin{align*}
&Hf(x,y)=\frac{\sqrt{x-1}}{\sqrt{2-x}}
\bigl((x-2)\partial_x f(x,y)+y\partial_y f(x,y)\bigr).
\end{align*}
Equality is attained for $f(x,y)=\mathsf P_n^{(\a,1)}(y)$.
\end{prop}
\begin{proof}
If $(x,y) \in \mathrm{T}$, then
\[
\widetilde{\Phi}_y(x,y)=
\begin{cases}
\sqrt{y(1-y)}, & \text{if } x \le 1,\\[6pt]
\dfrac{\sqrt{y}\,\sqrt{2-x-y}}{\sqrt{2-x}}, & \text{otherwise.}
\end{cases}
\]
Hence,
\begin{align*}
  \left\|\wt{\Phi}_y \partial_y f\right\|^2_{L^2(\mathrm{T},Y)} = \left\| \sqrt{y(1-y)} \partial_y f \right\|^2_{L^2([0,1]^2,Y)} + \left\|  \dfrac{\sqrt{y}\,\sqrt{2-x-y}}{\sqrt{2-x}} \partial_y f \right\|^2_{L^2(\triangle,Y)}.
\end{align*}
By Theorem \ref{BforSq},
\[
\left\| \sqrt{y(1-y)} \partial_y f \right\|^2_{L^2([0,1]^2,Y)} \leq n(n+\a +2) \left\|  f \right\|^2_{L^2([0,1]^2,Y)}.
\]
By the change of variable $x=s+1$, we obtain
\[
\left\|  \dfrac{\sqrt{y}\,\sqrt{2-x-y}}{\sqrt{2-x}} \partial_y f \right\|^2_{L^2(\triangle,Y)} =
\int_{0}^{1} \int_{0}^{1-y} \left|  \dfrac{\sqrt{y}\,\sqrt{1-s-y}}{\sqrt{1-s}} \partial_y f(s+1,y) \right|^2 y^\a \, \d s \, \d y,
\]
and
\[
\|Hf\|^2_{L^2(\triangle,Y)} = \int_{0}^{1} \int_{0}^{1-y} \left|  \frac{\sqrt{s}}{\sqrt{1-s}}
\Bigl((s-1)\partial_x f(s+1,y)+y\partial_y f(s+1,y)\Bigr) \right|^2 y^\a \, \d s \, \d y.
\]
Therefore, by Corollary 2.6 of \cite{GX}, we have
\[
\|Hf\|^2_{L^2(\triangle,Y)} + \left\|  \dfrac{\sqrt{y}\,\sqrt{2-x-y}}{\sqrt{2-x}} \partial_y f \right\|^2_{L^2(\triangle,Y)} \leq n(n+\a +2) \left\|  f \right\|^2_{L^2(\triangle,Y)}.
\]
Thus the inequality \eqref{DyT} holds. We now verify equality for $f(x,y)=\mathsf P_n^{(\a,1)}(y)$. By \eqref{eq:1d_eigs} and integration by parts, we have
\[
\int_{0}^{1} \left|\sqrt{y} \sqrt{1-y} \partial_y \mathsf{P}_n^{(\a,1)}(y)\right|^2 y^\a (1-y)  \, \d y = n(n+\a+2) \int_{0}^{1} \left|\mathsf{P}_n^{(\a,1)}(y)\right|^2 y^\a (1-y)  \, \d y.
\]
Hence,
\[
\left\| \sqrt{y(1-y)} \partial_y \mathsf{P}_n^{(\a,1)} \right\|^2_{L^2([0,1]^2,Y)} = n(n+\a +2) \|  \mathsf{P}_n^{(\a,1)} \|^2_{L^2([0,1]^2,Y)}.
\]
We examine the remaining part
\begin{align*}
&\|H \mathsf{P}_n^{(\a,1)} \|^2_{L^2(\triangle,Y)}  = \int_{0}^{1} \int_{1}^{2-y} \left| \frac{\sqrt{x-1}}{\sqrt{2-x}}
y\partial_y \mathsf{P}_n^{(\a,1)}(y)  \right|^2 y^\a \, \d x \, \d y, \\
&\|\wt{\Phi}_y \partial_y \mathsf{P}_n^{(\a,1)}\|^2_{L^2(\triangle,Y)} = \int_{0}^{1} \int_{1}^{2-y} \left|
\dfrac{\sqrt{y}\,\sqrt{2-x-y}}{\sqrt{2-x}} \partial_y \mathsf{P}_n^{(\a,1)}(y) \right|^2 y^\a  \, \d x \, \d y.
\end{align*}
Adding these identities and integrating with respect to $x$, we obtain
\begin{align*}
\|H \mathsf{P}_n^{(\a,1)} \|^2_{L^2(\triangle,Y)} + \|\wt{\Phi}_y \partial_y \mathsf{P}_n^{(\a,1)}\|^2_{L^2(\triangle,Y)}  =    \int_{0}^{1}  \left|
\sqrt{y(1-y)} \partial_y \mathsf{P}_n^{(\a,1)}(y) \right|^2 (1-y) y^\a \, \d y \\
 =n(n+\a+2) \int_{0}^{1} \left|\mathsf{P}_n^{(\a,1)}(y)\right|^2 y^\a (1-y)  \, \d y = n(n+\a+2) \| \mathsf{P}_n^{(\a,1)}\|^2_{L^2(\triangle,Y)}.
\end{align*}
Combining this identity with the corresponding equality on $[0,1]^2$ gives equality in \eqref{DyT}.
\end{proof}

\begin{lem}\label{lem:adjacent-jacobi-integral}
For \(\alpha>-1\) and \(n\ge 0\),
\[
\int_{-1}^1 (1-x)^\alpha \bigl(P_n^{(\alpha,1)}(x)\bigr)^2\, \d x
=
\frac{2^{\alpha+1}(n+1)}{n+\alpha+1}.
\]
\end{lem}

\begin{proof}
We start from the adjacent-parameter relation for Jacobi polynomials
(see, e.g., \cite[(18.9.5)]{NIST}):
\[
(2m+\alpha+1)P_m^{(\alpha,0)}(x)
=
(m+\alpha+1)P_m^{(\alpha,1)}(x)
+
(m+\alpha)P_{m-1}^{(\alpha,1)}(x),
\qquad m\ge 1.
\]
Equivalently,
\[
(m+\alpha+1)P_m^{(\alpha,1)}(x)
=
(2m+\alpha+1)P_m^{(\alpha,0)}(x)
-
(m+\alpha)P_{m-1}^{(\alpha,1)}(x).
\]
Iterating this identity for \(m=n,\ldots,1\) and using $P_0^{(\alpha,0)}=P_0^{(\alpha,1)}=1$, one obtains
\[
(n+\alpha+1)P_n^{(\alpha,1)}(x)
=
\sum_{k=0}^n (-1)^{\,n-k}(2k+\alpha+1)P_k^{(\alpha,0)}(x).
\]
Now square this identity and integrate on \([-1,1]\) with the weight \((1-x)^\alpha\). Since the family
\(\{P_k^{(\alpha,0)}\}_{k\ge 0}\) is orthogonal with respect to \((1-x)^\alpha\), all cross terms vanish, and therefore
\[
(n+\alpha+1)^2
\int_{-1}^1 (1-x)^\alpha \bigl(P_n^{(\alpha,1)}(x)\bigr)^2\, \d x
=
\sum_{k=0}^n (2k+\alpha+1)^2 h_k^{(\alpha,0)},
\]
where
\[
h_k^{(\alpha,0)}
:=
\int_{-1}^1 (1-x)^\alpha \bigl(P_k^{(\alpha,0)}(x)\bigr)^2\, \d x.
\]
By the standard Jacobi norm formula,
\[
h_k^{(\alpha,0)}=\frac{2^{\alpha+1}}{2k+\alpha+1}.
\]
Hence
\[
(n+\alpha+1)^2
\int_{-1}^1 (1-x)^\alpha \bigl(P_n^{(\alpha,1)}(x)\bigr)^2\, \d x
=
2^{\alpha+1}\sum_{k=0}^n (2k+\alpha+1).
\]
Therefore,
\[
(n+\alpha+1)^2
\int_{-1}^1 (1-x)^\alpha \bigl(P_n^{(\alpha,1)}(x)\bigr)^2\, \d x
=
2^{\alpha+1}(n+1)(n+\alpha+1),
\]
Thus
\[
\int_{-1}^1 (1-x)^\alpha \bigl(P_n^{(\alpha,1)}(x)\bigr)^2\, \d x
=
\frac{2^{\alpha+1}(n+1)}{n+\alpha+1},
\]
which completes the proof.
\end{proof}

\begin{lem}\label{lem:jacobi-moments}
Let $\a > -1$. Then
\[
\int_{0}^{1} (1- y) \left|
\partial_y \mathsf{P}_n^{(\a,1)}(y)  \right|^2 y^{\a+2} \,  \d y = \frac{n(2n+\a)(n+\a+2)}{2(2n+\a+2)} \int_{0}^{1} \left| \mathsf{P}_n^{(\a,1)}(y)\right|^2 y^{\a}  \, \d y,
\]
and
\[
\int_{0}^{1} \left| \mathsf{P}_n^{(\a,1)}(y)\right|^2 y^{\a+1}  \, \d y = \frac{2n+\a+1}{2n+\a+2}  \int_{0}^{1} \left| \mathsf{P}_n^{(\a,1)}(y)\right|^2 y^{\a}  \, \d y.
\]
\end{lem}

\begin{proof}
Set
\[
Q_n(y):=P_n^{(\alpha,1)}(1-2y)=\mathsf P_n^{(\alpha,1)}(y),
\]
and define
\[
I_n:=\int_0^1 y^{\alpha+1} (Q_n(y))^2\, \d y,
\qquad
A_n:=\int_0^1 (1-y)y^\alpha (Q_n(y))^2\, \d y.
\]
Then
\[
D_n:=\int_0^1 y^\alpha Q_n(y)^2\, \d y=I_n+A_n.
\]

We compute \(A_n\) and \(D_n\) separately. Using the change of variables
$
x=1-2y,
$
we obtain
\[
A_n
=
2^{-(\alpha+2)}
\int_{-1}^1 (1-x)^\alpha(1+x)\bigl(P_n^{(\alpha,1)}(x)\bigr)^2\, \d x.
\]
Now we apply the standard Jacobi norm formula
\[
\int_{-1}^1 (1-x)^\alpha(1+x)^\beta
\bigl(P_n^{(\alpha,\beta)}(x)\bigr)^2\, \d x
=
\frac{2^{\alpha+\beta+1}}{2n+\alpha+\beta+1}
\frac{\Gamma(n+\alpha+1)\Gamma(n+\beta+1)}
{n!\,\Gamma(n+\alpha+\beta+1)}.
\]
With \(\beta=1\), this gives
\[
A_n
=
2^{-(\alpha+2)}
\cdot
\frac{2^{\alpha+2}}{2n+\alpha+2}
\frac{\Gamma(n+\alpha+1)\Gamma(n+2)}
{n!\,\Gamma(n+\alpha+2)}.
\]
Since
$
\Gamma(z+1)=z\Gamma(z),
$
we get
\[
A_n=\frac{n+1}{(2n+\alpha+2)(n+\alpha+1)}.
\]
By the change of variables \(x=1-2y\) and Lemma~\ref{lem:adjacent-jacobi-integral}, we have
\[
D_n
=
2^{-(\alpha+1)}
\int_{-1}^1 (1-x)^\alpha \bigl(P_n^{(\alpha,1)}(x)\bigr)^2\, \d x = 2^{-(\alpha+1)}\cdot \frac{2^{\alpha+1}(n+1)}{n+\alpha+1}
=
\frac{n+1}{n+\alpha+1}.
\]
Since \(D_n=I_n+A_n\), we have
\[
I_n=D_n-A_n =  \frac{n+1}{n+\alpha+1}
\left(1-\frac{1}{2n+\alpha+2}\right).
\]
Hence
\[
I_n
=
\frac{2n+\alpha+1}{2n+\alpha+2}\cdot
\frac{n+1}{n+\alpha+1}
=
\frac{2n+\alpha+1}{2n+\alpha+2}\,D_n.
\]
This proves the second identity. It remains to establish the first one. Since $Q_n=\mathsf P_n^{(\alpha,1)}$, the Jacobi differential equation gives
\[
-\frac{\d}{\d y}
\left(y^{\alpha+1}(1-y)^2Q_n'(y)\right)
=
\lambda_n^{(\a,1)} y^\alpha(1-y)Q_n(y).
\]
Multiplying by $\dfrac{y}{1-y}Q_n(y)$ and integrating by parts, we obtain
\[
\lambda_n^{(\a,1)} I_n
=
\int_0^1 y^{\alpha+2}(1-y)|Q_n'(y)|^2\,\d y
+
\int_0^1 y^{\alpha+1}Q_n(y)Q_n'(y)\,\d y.
\]
Moreover, using $Q_n(1)=P_n^{(\alpha,1)}(-1)=(-1)^n(n+1)$,
\[
\int_0^1 y^{\alpha+1}Q_n(y)Q_n'(y)\,\d y
=
\frac{(n+1)^2}{2}-\frac{\alpha+1}{2}D_n.
\]
Consequently,
\[
\int_0^1 y^{\alpha+2}(1-y)|Q_n'(y)|^2\,\d y
=
\lambda_n^{(\a,1)} I_n-\frac{(n+1)^2}{2}+\frac{\alpha+1}{2}D_n.
\]
Using the formulas for $I_n$ and $D_n$ obtained above and simplifying, we get
\[
\int_0^1 y^{\alpha+2}(1-y)|Q_n'(y)|^2\,\d y
=
\frac{n(2n+\alpha)(n+\alpha+2)}{2(2n+\alpha+2)}D_n,
\]
which proves the first identity and completes the proof.
\end{proof}

\begin{cor}
Let $\widetilde{\Phi}_y$ be as in the above proposition. For $\a > -1$ let $\Upsilon(x,y)=y^\a$. Then
\begin{align*}
  \sqrt{\frac{n(n+\alpha+2)}{2}}  \le \sup_{\substack{f\in\Pi_n\\ f\ne0}} \frac{ {\|\widetilde{\Phi}_y \partial_y f\|}_{L^2(\mathrm{T},\Upsilon)} } { {\|f\|}_{L^2(\mathrm{T},\Upsilon)} }
\le \sqrt{n(n+\alpha+2)} .
\end{align*}
\end{cor}
\begin{proof}
For the upper bound, write $\mathrm T=Q\cup\triangle$,
where $Q=[0,1]^2$. On $Q$, the weight is
$\Upsilon=y^\a$, so the one-dimensional Jacobi inequality
in the $y$-variable has parameters $(\a,0)$ and gives
the constant $n(n+\a+1)$. On $\triangle$, the weights
$\Upsilon$ and $Y$ coincide, and the estimate established
in the proof of Proposition~\ref{prop:bernstein_y} applies
after dropping the nonnegative term involving $Hf$.
Consequently,
\begin{align*}
\|\wt\Phi_y\partial_y f\|_{L^2(\mathrm T,\Upsilon)}^2
&\le
n(n+\a+1)\|f\|_{L^2(Q,\Upsilon)}^2
+
n(n+\a+2)\|f\|_{L^2(\triangle,\Upsilon)}^2
\\
&\le
n(n+\a+2)\|f\|_{L^2(\mathrm T,\Upsilon)}^2.
\end{align*}
For the lower bound, take
$
Q_n(y):=\mathsf P_n^{(\a,1)}(y)$ and regard $Q_n$ as a polynomial on $\mathrm T$ independent of $x$. Direct
integration with respect to $x$ gives
\begin{align*}
\|\wt\Phi_y\partial_y Q_n\|_{L^2(\mathrm T,\Upsilon)}^2
&=
\int_0^1 y^{\a+1}
\bigl(2(1-y)+y\log y\bigr)|Q_n'(y)|^2\,\d y.
\end{align*}
For $0<y\le1$,
\[
\log y\ge \frac12\left(y-\frac1y\right),
\]
and hence
\[
2(1-y)+y\log y
\ge
\frac32(1-y)^2+y(1-y).
\]
Therefore,
\begin{align*}
\|\wt\Phi_y\partial_y Q_n\|_{L^2(\mathrm T,\Upsilon)}^2
\ge
\frac32
\int_0^1y^{\a+1}(1-y)^2|Q_n'(y)|^2\,\d y
+
\int_0^1y^{\a+2}(1-y)|Q_n'(y)|^2\,\d y.
\end{align*}
Set
\[
D_n:=\int_0^1 (Q_n(y))^2 y^\a\,\d y.
\]
By the Jacobi differential equation \eqref{eq:1d_eigs}, integration by parts, and Lemma~\ref{lem:jacobi-moments},
\[
\int_0^1 y^{\a+1}(1-y)^2|Q_n'(y)|^2\,\d y
=
\lambda_n^{(\a,1)} \int_0^1 |Q_n(y)|^2 y^\a(1-y)\,\d y
=
\frac{\lambda_n^{(\a,1)}}{2n+\a+2}D_n.
\]
Moreover, by Lemma~\ref{lem:jacobi-moments},
\[
\int_0^1y^{\a+2}(1-y)|Q_n'(y)|^2\,\d y
=
\frac{\lambda_n^{(\a,1)} (2n+\a)}{2(2n+\a+2)}D_n.
\]
Thus
\[
\|\wt\Phi_y\partial_y Q_n\|_{L^2(\mathrm T,\Upsilon)}^2
\ge
\frac{\lambda_n^{(\a,1)} (2n+\a+3)}{2(2n+\a+2)}D_n.
\]
On the other hand, the second identity in Lemma~\ref{lem:jacobi-moments}
yields
\begin{align*}
\|Q_n\|_{L^2(\mathrm T,\Upsilon)}^2
=
\int_0^1(2-y) |Q_n(y)|^2 y^\a\,\d y
=
\frac{2n+\a+3}{2n+\a+2}D_n.
\end{align*}
Consequently,
\[
\|\wt\Phi_y\partial_y Q_n\|_{L^2(\mathrm T,\Upsilon)}^2
\ge
\frac{\lambda_n^{(\a,1)}}{2}\|Q_n\|_{L^2(\mathrm T,\Upsilon)}^2,
\]
which proves the lower bound.
\end{proof}

\section{The Siciak extremal function}
Let $E$ be a compact subset of $\mathbb{C}^d$. Denote the uniform norm on $E$ by $\|\cdot\|_E$.
The Siciak extremal function on $E$, denoted by $\Phi_E(z)$, is defined for $z \in \mathbb{C}^d$ by
\begin{align}\label{extremPhi}
\Phi_{E}(z):=\sup \left\{|P(z)|^{\frac{1}{\operatorname{deg} P}}: \operatorname{deg} P \geq 1
\text { and } {\|P\|}_{E} \leq 1, \quad  P \in \CP\right\},
\end{align}
where $\CP$ denotes the space of holomorphic polynomials.
We refer to \cite{S1962} for properties of this function and its applications in the theory of analytic
functions in several complex variables. To state the result most relevant to us, we need the definition of {\it plurisubharmonic} (psh)
functions. A function $u$ with values in $[-\infty, +\infty)$ defined in an open set $X \subset \mathbb{C}^d$ is a psh function
if
\begin{enumerate}
    \item $u$ is upper semi-continuous;
    \item For arbitrary $z$ and $w$ in $\mathbb{C}^d$ the function
   \[
        \tau \mapsto u(z + \tau w)
  \]
is subharmonic in the open subset of $\mathbb{C}$ where it is defined.
\end{enumerate}
If both $u$ and $-u$ are plurisubharmonic, then $u$ is called \textit{pluriharmonic}.
The Lelong class of psh functions is defined by
$$
\mathfrak{L}_d := \left \{ u \ \text{psh on } \mathbb{C}^n: \ u(z) \leq \log(1 + \sqrt{|z_1|^2 + \ldots + |z_d|^2}) + O(1) \right \}.
$$
One of the basic results for Siciak's extremal function is its connection with the function
\begin{equation} \label{extremE}
V_E(z):= \sup \{ u(z) : u \in \mathfrak{L}_d, \, u|_E \le 0 \}.
\end{equation}
For $d =1$, the function $V_E$ is the classical Green function of the planar compact set $E$ that has a logarithmic
pole at infinity. The following theorem is due to Zakharyuta  \cite{Z} and Siciak \cite{S1981}.

\begin{thm}
If $E$ is a compact subset of $\mathbb{C}^d$ then
\[
\log \Phi_E(z) = V_E(z) \quad \text{for } z \in \mathbb{C}^d.
\]
\end{thm}
In particular, the above theorem is very helpful to compute $\Phi_E(z)$ in many concrete cases.
In the one-dimensional case we have (see \cite{Klimek})
\[
\Phi_{[-1,1]} (z) = h\left(\frac{1}{2}\lvert z - 1 \rvert + \frac{1}{2}\lvert z + 1 \rvert\right),
\qquad z \in \mathbb{C},
\]
where $h(\zeta)=\zeta+\sqrt{\zeta^2-1}$ if we choose a branch of the square root function so that $|h(\zeta)|>1$ for
$\zeta \in \mathbb{C}\setminus [-1,1]$. In particular, $h(t)=t+\sqrt{t^2-1}$ for $t \geq 1$.

If $E = I_d = [-1,1]^d$ is the $d$-dimensional cube, then, by the product property of the Siciak extremal function (see \cite{S1981}), we have
\[
\Phi_E(z) = \max_{1 \leq j \leq d} h\left(\frac{1}{2}\lvert z_j + 1 \rvert + \frac{1}{2}\lvert z_j - 1 \rvert\right), \quad z \in \mathbb{C}^d.
\]
If $E = S_d = \{x \in \mathbb{R}^d : x_1,\ldots,x_d \ge 0 \text{ and } x_1 + \cdots + x_d \le 1\}$ is the standard simplex, then Baran's formula (see \cite{Baran}) gives,
\begin{align*}
 \Phi_{S_d}(z)= h(|z_1| + \ldots + |z_d| + |z_1 + \ldots + z_d -1|).
\end{align*}
Since $\mathrm{T}=2S_2 \cap [0,2] \times [0,1]$, by \cite{B3}, for $z=(z_1,z_2) \in \mathbb{C}^2$, we have
\begin{align}\label{EforT}
  \Phi_{\mathrm{T}}(z)= \max \{ \Phi_{2S_2}(z), \Phi_{[0,2] \times [0,1]}(z)\} = \max \{ \Phi_{S_2}(z/2), \Phi_{I_2}(z_1-1,2z_2-1)\}.
\end{align}

Let $f$ be real-valued in a neighborhood of $x_0 \in \mathbb{R}$. The lower Dini derivative $D_{+}f$, also called a
lower right-hand derivative, of $f$ at $x_0$ is defined by
\begin{align*}
 D_{+}f(x_0):=\liminf\limits_{h \rightarrow 0^{+}} \frac{f(x_0+h) - f(x_0)}{h}.
\end{align*}
Let $\{e_1,\ldots,e_d\}$ be the standard orthonormal basis in $\mathbb{R}^d$. Let $E$ be a compact set in $\mathbb{C}^d$.
For $z \in \operatorname{int} E$ and each $j = 1, \ldots, d$, define $F_j : \mathbb{R} \to \mathbb{R}$ by
\[
F_j(t):=  V_E\bigl(z + \i\, t e_j\bigr), \quad t \in \mathbb{R}.
\]
Since $\Phi_E(z)=1$ for $z \in E$, it follows $V_E(z)=0$ by $\log \Phi_E(z) = V_E(z)$. Therefore, for $z \in \operatorname{int} E$,
\[
D_{+}F_j(0)= \liminf\limits_{\epsilon \rightarrow 0^{+}} \frac{V_E(z + \i\, \epsilon e_j) - V_E(z)}{\epsilon}
   = \liminf\limits_{\epsilon \rightarrow 0^{+}} \frac{V_E(z + \i\,\epsilon e_j)}{\epsilon}=:D^{+}_j V_E(z).
\]
By its definition, $D^{+}_j V_E$ can be called a Dini derivative of the extremal function $V_E$. These derivatives are
closely related to the Bernstein-type inequalities, as seen in the following theorem \cite{B1, B2}, in which a compact
set $K \subset \RR^d$ is treated as a subset of $\mathbb{C}^{d}$ such that $\mathbb{R}^{d}=\{(z_{1}, \ldots, z_{d}) \in \mathbb{C}^{d}: \operatorname{Im} z_{j}=0, j=1, \ldots, d\}$.

\begin{thm}
Let $K$ be a compact set in $\mathbb{R}^d$ with nonempty interior. For every $x \in \operatorname{int}K$ and every real polynomial $f$ in $d$ variables of total degree at most $n$,
\begin{align}\label{Baran_extremal}
  |\partial_j f(x)| \leq n D^{+}_j V_K(x) \left({\|f\|}^2_K - f^2(x)\right)^{1/2}, \quad j=1,\ldots,d.
\end{align}
\end{thm}
Related $L^p$ Bernstein inequalities with factors linked
to the pluripotential extremal function have been established
on certain cuspidal domains; see \cite{BT}.

To see the connection to the Bernstein inequality on the trapezoid, one can compute the Dini derivative explicitly,
\begin{align*}
      &D^{+}_i V_{2S_2}(x_1,x_2)= \frac{\sqrt{x_i+ 2- x_1-x_2} }{\sqrt{2x_i}\sqrt{2-x_1-x_2}}, \quad i = {1,2}, \\
      &D^{+}_1 V_{[0,2] \times [0,1]}(x_1,x_2)= \frac{1 }{\sqrt{x_1}\sqrt{2-x_1}}, \quad
      D^{+}_2 V_{[0,2] \times [0,1]}(x_1,x_2)= \frac{1 }{\sqrt{x_2}\sqrt{1-x_2}}.
\end{align*}
Since
$
V_{\mathrm{T}} = \max \{V_{2S_2}, V_{[0,2] \times [0,1]} \},
$
we obtain
\[
D_j^+ V_{\mathrm{T}}(x) = \max \left( D_j^+ V_{2S_2}(x), \, D_j^+ V_{[0,2] \times [0,1]}(x) \right).
\]
Hence, for $(x,y)$ belonging to the interior of the set $\mathrm{T}$,
\begin{align*}
D_1^+ V_{\mathrm{T}}(x,y) &=
\max \left\{
\frac{\sqrt{2 - y}}{\sqrt{2}\sqrt{x}\sqrt{2 - x - y}},
\;
\frac{1}{\sqrt{x(2 - x)}}
\right\} \\&= \frac{\sqrt{2 - y}}{\sqrt{2}\sqrt{x}\sqrt{2 - x - y}} = \frac{1}{\sqrt{2} \wt{\Phi}_x(x,y)}
\end{align*}
and
\[
D_2^+ V_{\mathrm{T}}(x,y) =
\max \left\{
\frac{\sqrt{2 - x}}{\sqrt{2}\sqrt{y}\sqrt{2 - x - y}},
\;
\frac{1}{\sqrt{y(1 - y)}}
\right\}.
\]
Thus $\wt\Phi_x$ is, up to the constant factor $1/\sqrt2$, the reciprocal
of the corresponding Dini derivative. In the $y$-direction,
\[
\frac{1}{D_2^+V_{\mathrm T}(x,y)}
=
\min\left\{
\sqrt{2}\,\frac{\sqrt y\sqrt{2-x-y}}{\sqrt{2-x}},
\sqrt{y(1-y)}
\right\}.
\]
Since
\[
\wt\Phi_y(x,y)
=
\min\left\{
\frac{\sqrt y\sqrt{2-x-y}}{\sqrt{2-x}},
\sqrt{y(1-y)}
\right\},
\]
we obtain
\[
\wt\Phi_y(x,y)
\le
\frac{1}{D_2^+V_{\mathrm T}(x,y)}
\le
\sqrt2\,\wt\Phi_y(x,y).
\]
Consequently, the boundary factors appearing in \eqref{DxT} and
Proposition~\ref{prop:bernstein_y} are directly controlled by the reciprocal
Dini derivatives entering \eqref{Baran_extremal}. Thus, (5.3) provides a weighted $L^2$ analogue featuring the same geometric factor, up to normalization, as the pointwise Bernstein inequality \eqref{Baran_extremal} for $K= \mathrm{T}$.

\section*{Acknowledgements}
The author was supported by the Polish National Science Centre (NCN) Miniatura grant no. 2025/09/X/ST1/00082.

ChatGPT was used only for language editing of the manuscript. All mathematical results, proofs, and arguments were developed and verified independently by the author.

\end{document}